\documentclass[11pt]{amsart}
\usepackage{amsmath,amsthm, amscd, amssymb, amsfonts}
\usepackage{hhline}

\usepackage[all]{xy}

\usepackage{pictexwd} 
\usepackage{pslatex}
\usepackage{t1enc}
\usepackage{latexsym}
\usepackage[latin1]{inputenc} 
\usepackage{subfigure}    
\usepackage{verbatim}

\newcommand\boxe{\begin{tabular}{|p{0,1cm}|}
\hline \\ \hline \end{tabular}}

\newcommand\boxur{\begin{tabular}{p{0,1cm}|}
\hline \\  \end{tabular}}

\newcommand{\acts}{{\rightharpoondown}}

\newcommand\id{\operatorname{id}}
\newcommand\idd{\mathbf{id}}

\newcommand{\G}{{\mathcal G}}

\newcommand{\D}{{\mathcal D}}
\newcommand{\Ec}{{\mathbf E}}

\newcommand\tto{\rightrightarrows}

\newcommand{\B}{{\mathcal B}}

\newcommand{\Hc}{{\mathcal H}}
\newcommand{\Vc}{{\mathcal V}}
\newcommand{\Pc}{{\mathcal P}}
\newcommand{\Oc}{{\mathcal O}}

\numberwithin{equation}{section}\theoremstyle{plain}

\newtheorem{theorem}{Theorem}[section]
\newtheorem{lema}[theorem]{Lemma}

\newtheorem{prop}[theorem]{Proposition}

\theoremstyle{definition}
\newtheorem{definition}[theorem]{Definition}

\theoremstyle{remark}
\newtheorem{obs}[theorem]{Remark}

\newcommand\iddv{\mathbf{id}}

\def\pf{\begin{proof}}
\def\epf{\end{proof}}

\theoremstyle{remark}

\def\h4n{\hspace{-0.4cm}}

\def\N{I\hspace{-0.8ex} N}

\def\[{|\hspace{-0.2ex} [}
\def\]{]\hspace{-0.2ex} |}

\renewcommand\h{\frak h}

\newcommand\pfibrado[2]{\;{}_{#1}\hspace{-2pt}\times_{#2}\;} 

 \newcommand\caja[5]{
 \begin{tabular}{c}
 \begin{picture}(50,60)(0,0) 
 \put(5,8){\framebox(40,40){#1}}
 \put(25,55){\makebox(0,0)[c]{${\scriptstyle \text{#2}}$}} 
 \put(48,28){\makebox(0,0)[l]{${\scriptstyle \text{#3}}$}} 
 \put(25,0){\makebox(0,0)[c]{${\scriptstyle \text{#4}}$}}  
 \put(3,28){\makebox(0,0)[r]{${\scriptstyle \text{#5}}$}} 
 \end{picture}
 \end{tabular}
 }

\newcommand\cajaMedium[5]{
 \begin{tabular}{c}
 \begin{picture}(37,50) 
 \put(3,8){\framebox(30,30){#1}}
 \put(18,43){\makebox(0,0)[c]{${\scriptstyle \text{#2}}$}} 
 \put(37,23){\makebox(0,0)[l]{${\scriptstyle \text{#3}}$}} 
 \put(18,1){\makebox(0,0)[c]{${\scriptstyle \text{#4}}$}}  
 \put(1,23){\makebox(0,0)[r]{${\scriptstyle \text{#5}}$}} 
 \end{picture}
 \end{tabular}
 }

\newcommand\cajaSmall[5]{
 \begin{tabular}{c}
 \begin{picture}(25,40) 
 \put(3,8){\framebox(20,20){#1}}
 \put(11,32){\makebox(0,0)[c]{${\scriptstyle \text{#2}}$}} 
 \put(25,18){\makebox(0,0)[l]{${\scriptstyle \text{#3}}$}} 
 \put(12,2){\makebox(0,0)[c]{${\scriptstyle \text{#4}}$}}  
 \put(0,18){\makebox(0,0)[r]{${\scriptstyle \text{#5}}$}} 
 \end{picture}
 \end{tabular}
 }

\begin{document}

\renewcommand{\baselinestretch}{1.2}
\thispagestyle{empty}

\title[Double Lie groupoids]{On slim double Lie groupoids}
\author[Andruskiewitsch, Ochoa and Tiraboschi]{Nicol\'as Andruskiewitsch}
\address{\noindent
Facultad de Matem\'atica, Astronom\'\i a y F\'\i sica, Universidad
Nacional de C\'ordoba.  CIEM -- CONICET. (5000) Ciudad
Universitaria, C\'ordoba, Argentina}
 \email{andrus@mate.uncor.edu,  \emph{URL:}\/ http://www.mate.uncor.edu/andrus}
 \author[]{Jesus Alonso Ochoa Arango}
 \email{jeochoa@famaf.unc.edu.ar}
\author[]{Alejandro Tiraboschi}
 \email{tirabo@famaf.unc.edu.ar}
 \thanks{This work was partially supported by CONICET, Fundaci\' on
Antorchas, Agencia C\'ordoba Ciencia, ANPCyT    and Secyt (UNC)}
\date{\today}
\begin{abstract} We prove that
every slim double Lie  groupoid with proper core action is
completely determined by a factorization of a certain canonically
defined "diagonal" Lie groupoid.
\end{abstract}

\maketitle

\section*{Introduction}

A double groupoid is a set $\B$ provided with two different
but compatible grou-poid structures.
It is useful to represent the elements of $\B$ as boxes that merge
horizontally or vertically according to the
groupoid multiplication into consideration.
The vertical (respectively horizontal) sides of a box belong
to another groupoid $\Vc$ (resp. $\Hc$).
A double groupoid is \emph{slim} if any box is determined by its four sides.
The notion of double groupoids was introduced by Ehresmann \cite{ehr},
and later studied in \cite{brown, bj, BM, bs} and references therein.

The notion of double Lie groupoid was defined and investigated by
K. Mackenzie \cite{mk1, mk2}; see also \cite{p, mk3, wl}
for applications to differential and Poisson geometry.
In particular the question of the classification of double
Lie groupoids was raised in \cite{mk1}, see also \cite{BM}.
In the latter article, a complete answer was given in the restricted
case of locally trivial double Lie groupoids.  More recently, a
description in two stages of discrete double groupoids was given
in \cite{AN3}. To state them, let us recall that a diagram
over a pair of groupoids $\Vc$ and $\Hc$ is a triple $(\D, j, i)$
where $\D$ is a groupoid  and $i: \Hc \to \D$, $j: \Vc
\to \D$ are morphisms of groupoids (over a fixed set of points).
The stages in  \cite{AN3} are:
\begin{enumerate}
 \item[(a)] Any double groupoid is an extension of slime double
 groupoid (its \emph{frame}) by an abelian group bundle.
\item[(b)] The category of slim double groupoids, with fixed vertical
and horizontal goupoids $\Vc$ and $\Hc$, satisfying the filling condition,
is equivalent to the category of diagrams over $\Vc$ and $\Hc$.
\end{enumerate}

In this paper,  we extend stage (b) to the setting of double Lie
groupoids. In this context, instead of the filling condition,
one requires that the double source map is a surjective submersion \cite{mk1}.
As one may naturally expect, there are some topological and
geometrical ingredients in our main Theorem \ref{categoryequivalence}, which says:

\smallbreak
\emph{The category of slim double Lie groupoids, with fixed vertical
and horizontal Lie groupoids $\Vc$ and $\Hc$, and proper core action,
is equivalent to the category of diagrams of Lie groupoids $(\D, j, i)$
such that the maps $j$ and $i$ are transversal at the identities.}

Our proof of this theorem relies on \cite[Theorem 2.8]{AN3} and some
topological and differentiable considerations such as properness
of the \emph{core action} on one side and a transversality
condition on the morphisms involved in a diagram of Lie groupoids on the other.
It is also possible to adjust stage \emph{(a)} to the context
of double Lie groupoids but we postpone the investigations to a
future paper.

\section{Preliminaries on Lie groupoids an double Lie groupoids}\label{prels}

We denote a groupoid in the form $ \xymatrix{\G \ar@<2pt>[r]^{s}
\ar@<-2pt>[r]_{e} &\Pc}$, where $s$ and $e$ stand for `source' and
`end' respectively; and the identity map is $\id: \Pc \to \G$.
Recall that a groupoid $ \xymatrix{\G \ar@<2pt>[r]^{s}
\ar@<-2pt>[r]_{e} & \Pc} $ is a \textit{Lie groupoid} \cite{mk4}, if $\Pc$
and $\G$ are smooth manifolds, $s$ and $e$ are surjective
submersions and the other structural maps are smooth. The {\em
anchor} of $\G$   is the map $\chi:\G \to \Pc \times \Pc$ given by
$\chi(g) = (s(g),e(g) )$.

\medbreak

We recall the following well known definition.
\begin{definition}
A {\em left action} of a groupoid $\G \tto \Pc $ {\em along}
a map $\epsilon:N \to \Pc$ is given by a map
$ G \pfibrado{e}{\epsilon}\N \to N$, denoted by $(g,n) \mapsto gn$,
which satisfies the following identities:
$$\epsilon(hy) = s(h),\quad \id(\epsilon(y))\; y = y, \quad (gh)y = g(hy),$$
for all $g,h \in \G$ and $y \in N$ such that $e(g) = s(h)$
and $e(h) = \epsilon(y)$.

The \emph{transformation} or \emph{action groupoid}
$\G \ltimes N \rightrightarrows N$, associated with such
an action, is the groupoid with set of arrows $\G \pfibrado{e}{\epsilon} N$
and base $N$. The source and target maps are
$$ s',\; e': \G \ltimes N \to N, \quad \text{given by} \quad s'(g, n) = gn
\quad \text{and} \quad e'(g, n) = n,$$
respectively, and composition $(g, n)(h, m) = (gh, m)$.

\end{definition}

\begin{obs}
If in the above definition $\G \tto \Pc $ is a Lie groupoid, $N$ a smooth
manifold and $\epsilon:N \to \Pc$ a smooth map, we define a left action
by the same properties and the only extra requirement  is the smoothness of the map that
gives the action. The resulting action groupoid is again a Lie groupoid.

\end{obs}
\smallbreak
We remind now the definition of local bisections on a Lie groupoid.
\begin{definition}
Let $\xymatrix{\G \ar@<2pt>[r]^{s} \ar@<-2pt>[r]_{e} &\Pc}$ be a Lie groupoid
and let $U \subseteq \Pc$ be an open subset. A \emph{local bisection} of $G$ on $U$ is a
smooth map $\sigma: U \to \G$ which is a section of $e$ such that
$V:= (s \circ \sigma)(U)$ is an open subset of $\Pc$ and
$s \circ \sigma : U \to V$ is a diffeomorphism.
Define $\G^U = s^{-1}(U)$ and $\G_U = e^{-1}(U)$.
The \emph{local left} and \emph{right translations} induced by $\sigma$
are (respectively) the maps
\begin{equation}
\begin{aligned}
L_{\sigma}&: \G^U \to \G^V, \quad g \mapsto \sigma(s(g))g;\quad \text{and} \\
R_{\sigma}&: \G_V \to \G_U, \quad g \mapsto g\sigma((s \circ \sigma)^{-1}(e(g))).
\end{aligned}
\end{equation}

\end{definition}
For more on bisections see \cite[Section 1.4]{mk4}.

\begin{subsection}{Double Lie groupoids}

\begin{definition}[Ehresmann]

A \textit{double groupoid} is a groupoid object internal to the
category of groupoids. That is, a \textit{double groupoid} consist
of a set $\B$ with two groupoid structures with \textit{bases}
$\Hc$ and $\Vc$, which are themselves groupoids over a common base
$\Pc$, all subject to the compatibility condition that the
structure maps of each structure are morphisms with respect to the
other.
\end{definition}

It is usual to represent a double groupoid $(\B; \Vc, \Hc; \Pc)$
as a diagram of four related groupoids
$$
\xymatrix{ \B \ar@<2pt>[rr]^{l} \ar@<-2pt>[rr]_{r}
\ar@<2pt>[d]^{b} \ar@<-2pt>[d]_{t}
& & \Vc \ar@<2pt>[d]^{b} \ar@<-2pt>[d]_{t}\\
\Hc  \ar@<2pt>[rr]^{l} \ar@<-2pt>[rr]_{r} & & \Pc  }
$$
where $t$, $b$, $l$, $r$ mean `top', `bottom', `left' and `right',
respectively. We sketch the main axioms that these groupoids
should satisfy and refer \emph{e.~g.} to \cite[Section 2]{AN1} and
\cite[Section 1]{AN2} for a detailed exposition and other
conventions.

The elements of $\B$ are called `boxes' and will be denoted by
$$
A = \quad\caja{$A$}{$t(A)$}{$r(A)$}{$b(A)$}{$l(A)$}\quad \in\B.
$$
Here $t(A),\;b(A) \in \Hc$ and $l(A),\;r(A) \in \Vc$.
The identity maps will be denoted $\idd: \Vc \to \B$ and $\idd: \Hc
\to \B$. The product in the groupoid $\B$ with base $\Vc$ is
called {\em horizontal } and denoted by $AB$ or $\{AB\}$, for $A,B
\in \B$ with $r(A) = l(B)$. The product in the groupoid $\B$ with
base $\Hc$ is called {\em vertical} and denoted by $\begin{matrix}
A\\B\end{matrix}$ or $\left\{\begin{matrix}
A\\B\end{matrix}\right\}$, for $A,B \in \B$ with $l(A) =t(B)$.
This pictorial notation is useful to understand the products in
the double structure. For instance, compatibility axioms between
the horizontal and vertical products are described by

$$
\cajaMedium{$A$}{$t$}{$r$}{$b$}{$l$} \;
\cajaMedium{$B$}{$t'$}{$r'$}{$b'$}{$r$} =
\cajaMedium{$\{AB\}$}{$tt'$}{$r'$}{$bb'$}{$l$}\quad \text{ and }
\quad
\begin{matrix}
\cajaMedium{$A$}{$t$}{$r$}{$b$}{$l$} \\
\cajaMedium{$B$}{$b$}{$r'$}{$b'$}{$l'$}
\end{matrix} =\;
\cajaMedium{$\scriptstyle{\left\{\begin{matrix} A
\\B\end{matrix}\right\}}$}{$t$}{$rr'$}{$b'$}{$ll'$}
$$
We omit the letter inside the box if no confusion arises. We also
write $A^h$  and $A^v$ to denote the inverse of $A\in \B$ with
respect to the horizontal and vertical structures of  groupoid
over $\B$ respectively. When one of the sides of a box is an
identity, we draw this side as a double edge. For example, if
$t(A) = \id_p$, we draw \begin{tabular}{|p{0,1cm}|} \hhline{|=|} \\
\hline\end{tabular}  and say that  $t(A) \in \Pc$.

\begin{definition}[Mackenzie, \cite{mk1}]
A double groupoid is a \textit{double Lie groupoid} if all four
groupoids involved are Lie groupoids and the \textit{double source
map}
$$
\textbf{S}: \B \to \Hc \pfibrado{l}{t} \Vc, \qquad A \mapsto
\textbf{S}(A) = (t(A), l(A)),
$$
is a surjective submersion.
\end{definition}

For clarity, we shall say that a double groupoid is
\emph{discrete} if no Lie structure is present. A  discrete double
groupoid satisfies the \textit{filling condition} when the double
source map defined above is surjective. We refer the reader to \cite{AN3} for details.

\begin{definition}[Brown and Mackenzie, \cite{BM,
mk1}] Let $(\B;\Vc,\Hc;\Pc)$ be a double Lie groupoid. The core
groupoid $\Ec(\B)$ of $\B$  is
$$\Ec(\B)=\{E \in \B : \; t(E), \;
r(E) \in \Pc\}$$ with $s_\Ec$, $e_\Ec: \Ec(\B) \to \Pc$,
$s_\Ec(E)=bl(E)$, $e_\Ec(E)= tr(E)$; identity map given by $\id_p=
\Theta_p := \idd \circ \id(p)$; multiplication and inverse given
by
\begin{equation}\label{productcore} E \circ F : =
\left\{\begin{matrix} {\scriptstyle \idd l(F)} & F\\
E & {\scriptstyle \idd(b(F)) \vspace{-1pt}} \end{matrix} \right\},
\qquad E^{(-1)}: = (E \iddv b(E)^{-1})^v
= \left\{\begin{matrix}\iddv l(E)^{-1} \vspace{-4pt}\\
E^h\end{matrix} \right\},
\end{equation}
$E , F\in \Ec(\B)$. That is, the elements of $\Ec(\B)$ are of the
form $E =
\begin{tabular}{|p{0,1cm}||} \hhline{|=||} \\ \hline \end{tabular}
$; the source gives the bottom-left vertex and the target  gives
the top-right vertex of the box.
Clearly $s_\Ec$ and $e_\Ec$ are surjective submersions. Thus
$\Ec(\B)$ becomes a Lie groupoid,
differentiability conditions being easily verified because
$\Ec(\B)$ is an embedded submanifold of $\B$.
\end{definition}

\end{subsection}

\begin{subsection}{Coarse double groupoid}

Let $\Pc$ be a set and $\Vc$, $\Hc$ be groupoids over $\Pc$. Let $\square(\Vc, \Hc)$ be the set
$(\Vc\pfibrado{b}{l}\Hc)\pfibrado{(t,r)}{(l,b)}(\Hc\pfibrado{t}{r}\Vc)$;
that is, $\square(\Vc, \Hc)$ is the set of quadruples $\begin{pmatrix} \quad x  \quad \\  f \quad g \\
\quad y \quad\end{pmatrix}$ with $x,y\in \Hc$, $f,g\in \Vc$ such
that
\begin{equation*}
l(x) = t(f), \quad r(x) = t(g), \quad l(y) = b(f), \quad r(y) =
b(g). \end{equation*} If no confusion arises, we shall denote a
quadruple as above by  a box $\begin{matrix} \quad x \quad \\
h\,\, \boxe \,\, g \\ \quad y \quad
\end{matrix}$. The collection $\begin{matrix} \Box(\Vc, \Hc) &\rightrightarrows
&\Hc \\\downdownarrows &&\downdownarrows \\ \Vc &\rightrightarrows
&\Pc \end{matrix}$ forms a double groupoid in the obvious way, called the \textit{coarse double groupoid} with sides in $\Hc$ and $\Vc$.
\begin{obs}
\noindent (i) Let $M, N$ and $P$ be smooth manifolds, let $f:M \to P$ and $g:N \to P$ be smooth maps,
we remind that $f$ and $g$ are called \emph{transversal} at $p = f(m)=g(n)$,
for $m \in M$, $n \in N$, if $(T_m f)(T_m M)+(T_n g)(T_n N) = T_p P.$
We said that $f$ and $g$ are \emph{transversal} if they are  transversal at any $p$ as above.

\smallbreak
\noindent (ii) Let $\Vc$ and $\Hc$ be Lie groupoids over the same manifold $\Pc$, then if the anchors maps $\chi_{\Vc}:\Vc \to \Pc \times \Pc $ or $\chi_{\Hc} : \Hc \to \Pc \times \Pc$ are transversal then $\square(\Vc, \Hc)$ is a double Lie groupoid \cite{BM}.
\end{obs}
\begin{definition}  \cite{AN3}
A double groupoid $(\B; \Vc, \Hc; \Pc)$ is \textit{slim} if the morphism of (discrete) double groupoids
$\Pi: \B \to \Box(\Vc, \Hc)$ given by
\begin{center}
$\Pi\left(\cajaSmall{$A$}{$x$}{$g$}{$y$}{$f$} \right) = \begin{pmatrix} \quad x \quad
\\  f \quad g \\ \quad y \quad\end{pmatrix}, \qquad
\cajaSmall{$A$}{$x$}{$g$}{$y$}{$f$}
\in \B,$
\end{center}
is injective.
\end{definition}

\end{subsection}

\section{Diagrams of Groupoids}\label{gr-diag}

\begin{definition}\cite{AN3}
Let $\Vc$ and $\Hc$ be groupoids over the same base $\Pc$. A
\emph{diagram}  over $\Hc$ and $\Vc$ is a triple $(\D, j, i)$
where $\D$ is a groupoid over $\Pc$ and $i: \Hc \to \D$, $j: \Vc
\to \D$ are morphisms of groupoids over $\Pc$.

If $\Vc$ and $\Hc$ are Lie groupoids, then a \textit{diagram of
Lie groupoids} over $\Hc$ and $\Vc$ is a diagram of groupoids,
such that $\D$ is a Lie groupoid and $i$, $j$ are smooth.
\end{definition}

\bigbreak

To each diagram of groupoids we can associate a discrete double groupoid,
denoted by $\square(\D, j, i)$ and defined as follows: the boxes in
$\square(\D, j,i)$ are of the form
$$A =
\begin{matrix} \quad x \quad \\ h \,\, \boxe \,\, g \\ \quad y\quad
\end{matrix} \in \square(\Vc, \Hc),$$ with $x, y \in \Hc$, $g, h \in \Vc$, such that $$i(x)j(g) = j(h) i(y)\quad \text{ in } \D.$$

\begin{definition}\cite{AN3}
A diagram of groupoids $(\D, j, i)$, over $\Vc$ and $\Hc$ is
called a $(\Vc, \Hc)$-\textit{factorization} of $\D$, if
$\D = j(\Vc)i(\Hc)$.
\end{definition}

\medbreak
Our aim is to determine when $\square(\D, j, i)$ is a double Lie groupoid.
We define two maps. The first one is the composition $\Hc\pfibrado{r}{t}\Vc \overset{i\times
j}\longrightarrow \D \pfibrado{e}{s}\D \overset{m}
\longrightarrow\D$, i. e. $$\Phi: \Hc\pfibrado{r}{t}\Vc \to\D, \quad (x,g)
\mapsto i(x)j(g),$$
where $e$, $s$ and $m$ are the end, source and multiplication maps of $\D$.
The second one is $$\Psi:  \Vc\pfibrado{b}{l}\Hc \to\D, \qquad (f,y) \mapsto j(f)\;i(y).$$
Since $t,\; b,\; l$ and $r$ are surjective
submersions, we have that the fiber products involved in the
above maps, $\Vc \pfibrado{b}{l}\Hc$ and
$\Hc\pfibrado{r}{t}\Vc$ are embedded submanifolds of
$\Vc\times\Hc$ and $\Hc\times\Vc$ respectively, and since $i$
and $j$ are smooth then $\Phi$ and $\Psi$ are also smooth. With the above maps
\begin{equation*}
\square(\D, j,i)=(\Vc\pfibrado{b}{l}\Hc)\pfibrado{\Psi}{\Phi}(\Hc\pfibrado{r}{t}
\Vc),
\end{equation*}

\noindent and from general theory of transversality \cite[Prop.
2.5]{Lang} if $\Phi$ and $\Psi$ are transversal, then $\square(\D,
j,i)$ is an embedded submanifold of
$(\Vc\pfibrado{b}{l}\Hc)\times(\Hc\pfibrado{r}{t}\Vc)$.

\begin{lema}\label{cudradovariedad}
Let $(\D,j,i)$ be a diagram of Lie groupoids. If $i$ and
$j$ are transversal at the identities, then $\Psi$ and $\Phi$ defined above are submersions.
\end{lema}
\pf We take $((f,
y),(x,g))\in(\Vc\pfibrado{b}{r}\Hc)\times(\Hc\pfibrado{r}{t}\Vc)$
such that $\Psi(f, y)=\Phi(x, g)$ i.e, $j(f)i(y)=i(x)j(g)$. Now, by \cite[Prop. 2.5]{Lang}, we have
$$T_{(x,g)}(\Hc\pfibrado{r}{t}\Vc)=\{(Y, X)\in (T_x\Hc)\times(T_g\Vc)/
(T_x \;r)(Y)=(T_g \;t)(X)\}.$$
Let $W\in T_{(x,g)}(\Hc\pfibrado{r}{t}\Vc)$. We need
to prove that there is $(X_1, Y_1)$ belonging to $T_{(f,y)}(\Vc\pfibrado{b}{l} \Hc)$ such that
\begin{align*}
T_{(f, y)}\Psi(X_1,Y_1) &=(T_{(j(f), i(y))}m)(T_{(f, y)}j\times i)(X_1,Y_1) \\
  &= (T_{(j(f),i(y))}m)((T_fj)(X_1), (T_y\,i)(Y_1)).
  \end{align*}
We know that in $\Hc\rightrightarrows \Pc$ there is a local
bisection $\tau: U\to \Hc$ with $r(y)\in U\subseteq\Pc$ open and
$\tau(r(y))=y$ \cite[Prop. 1.4.9]{mk4}. Since $\tau$ is a
bisection, it induces local left and right translations defined as
follows. Set $V= (l\circ \tau)(U)$, open in $\Pc$,
$\Hc^U = l^{-1}(U)$ and $\Hc_U = r^{-1}(U)$ (the same for $V$), and
$$L_{\tau}: \Hc^U \to \Hc^V,\quad z  \mapsto \tau(l(z))z \quad
\text{and} \quad R_{\tau}: \Hc_V \to \Hc_U, \quad z  \mapsto z\;
\tau((l \circ \tau)^{-1} r(z))$$
Define the map $\tau_\D: U \to \D$ by $i \circ \tau$ and using
that $i$ is a groupoid morphism note that it is a local bisection
of $\D$. Also note that $\tau_\D(e(i(y)))=(i\circ
\tau)(r(y))=i(y)$.

In the same way there is a local bisection $\sigma:U' \to \Vc$
such that $\sigma_\Vc(b(f))=f$ with $U'\subseteq \D$ open and
$b(f) \in U'$. Again this induces a bisection in $\D$, $\sigma_\D:
U' \to \D$ such that $\sigma_\D(e(j(f)))=j(f)$.
Let $(X_1, Y_1)\in T_{(f,y)}(\Vc\pfibrado{b}{l} \Hc)$. Then by Xu's
formula for product in a tangent groupoid \cite[Theorem
1.4.14]{mk4} we obtain:
\begin{align*}
T_{(f, y)}\Psi(X_1, Y_1) &=(T_{(j(f), i(y))}m)((T_fj)(X_1),(T_y\,i)(Y_1))\\
  &= (T_{i(y)}L_{\sigma_\D})(T_y\,i)(Y_1)+(T_{j(f)}R_{\tau_\D})(T_fj)(X_1)
  \\ &\;-(T_{i(y)}L_{\sigma_\D})(T_{\id_{\D}l(y)}R_{\tau_\D})(T_{l(y)}\id_\D)(z)
 ,  \end{align*}
where we write
$z=(T_{j(f)}e)(T_fj)(X_1)=(T_{i(y)}s)(T_y\,i)(Y_1).$

Now $(T_{j(f)i(y)}L^{-1}_{\sigma_\D})(W)\in T_{i(y)}\D$ because we
have
\begin{align*}
L^{-1}_{\sigma_\D}(j(f)i(y)) &=
\sigma_\D^{-1}(s(j(f)i(y)))j(f)i(y)\\ &= \sigma_\D((s \circ
\sigma_\D)^{-1}(s (j(f)i(y))))^{-1}j(f)i(y)\\ &=
\sigma_\D((s \circ \sigma_\D)^{-1} s (j(f)))^{-1}j(f)i(y)\\
&= \sigma_\D(e (j(f)))^{-1}j(f)i(y)\\ &=
j(f)^{-1}j(f)i(y)=i(y).
\end{align*}
In analogous way, we have
$(T_{i(y)}R_{\tau_\D}^{-1})(T_{j(f)i(y)}L_{\sigma_\D}^{-1})(W)\in
T_{Id_\D l(y)} \D$ since
\begin{align*}
R_{\tau_\D}^{-1}(i(y)) &= i(y)\tau_\D^{-1}((s \circ
\tau_\D^{-1})^{-1} e (i(y)))\\ &= i(y)\tau_\D((s \circ
\tau_\D)^{-1}((s \circ \tau_\D^{-1})^{-1} e (i(y))))^{-1}\\ &= i(y)\tau_\D((s \circ
\tau_\D)^{-1}((s \circ \tau_\D) e (i(y))))^{-1}\\
&=i(y)\tau_\D(e (i(y)))^{-1} = i(y)^{-1}i(y)=Id_\D s (i(y)) =
Id_\D (l (y)).
\end{align*}
Denote $p=l(y)$ since $i$ and $j$ are transversal
at $Id_\D (p)$ then
$$T_{Id_p}\D = (T_{Id_p}i)(T_{ Id_p}\Hc)+(T_{ Id_p}j)(T_{Id_p}\Vc)$$
and in consequence we can find $X\in T_{Id_p}\Hc$, $ Y \in
T_{Id_p}\Vc$ such that
$$(T_{i(y)}R_{\tau_\D}^{-1})(T_{j(f)i(y)}L_{\sigma_\D}^{-1})(W) = (T_{Id_p}i)(X)+(T_{Id_p}j)(Y).$$
Thus, if we consider the vectors
\begin{align*}
X' = X + (T_p Id_\Hc)(T_{Id_p} t)(Y),\quad Y' = Y + (T_p
Id_\Vc)(T_{Id_p}l)(X),
\end{align*}
a direct calculation shows that
\begin{equation}\label{tr}(T_{i(y)}R_{\tau_\D}^{-1})(T_{j(f)i(y)}L_{\sigma_\D}^{-1})(W)
= (T_{Id_p}i)(X')+(T_{Id_p}j)(Y')-(T_p Id_\D)(Z)\end{equation}
where $Z=(T_{Id_p} b)(Y)+(T_{Id_p}l)(X)$. Since
\begin{align*}
(L_{\sigma_\D}\circ R_{\tau_\D} \circ i) = (L_{\sigma_\D} \circ i
\circ R_{\tau}), \quad (L_{\sigma_\D}\circ R_{\tau_\D}\circ j) =
(R_{\tau_\D}\circ j \circ L_{\sigma_\Vc}),
\end{align*}
we may apply $(T_{i(y)}L_{\sigma_D})(T_{Id_p}R_{\tau_\D})$ to both
sides of \eqref{tr}, and arrive to
$$W = (T_{i(y)}L_{\sigma_\D})(T_y\,i)(Y_1)+(T_{j(f)}R_{\tau_\D})(T_fj)(X_1)
-(T_{i(y)}L_{\sigma_\D})(T_{Id_p}R_{\tau_\D})(T_p Id_\D)(Z),$$
where $X_1 = (T_{Id_p}L_{\sigma_\Vc})(Y')$, $Y_1 =
(T_{Id_p}R_{\tau})(X')$. It is clear that $(T_f \;b)(X_1) = Z =
(T_y \;l)(Y_1),$ thus $T_{(f, y)}\Psi(X_1, Y_1) = W$.

We prove that $\Phi$ is a submersion in the same way. \epf

From the above result we obtain the following immediate
consequence.

\begin{theorem}\label{associatedslim}
Let $(\D, j, i)$ be a $(\Vc, \Hc)$-\textit{factorization} of the Lie
groupoid $\D$. If $i$ and $j$ are transversal
at the identities, then $\square(\D, j, i)$ is an slim double Lie
groupoid.
\end{theorem}
\begin{proof} By Lemma \ref{cudradovariedad} we have  that $\Phi$ and $\Psi$
are transversal, thus $\square(\D, j,i)$ is an embedded
submanifold of
$(\Vc\pfibrado{b}{l}\Hc)\times(\Hc\pfibrado{r}{t}\Vc)$. Since
$\Phi$ and $\Psi$ are surjective submersions, both projections of
$ \square(\D,
j,i)=(\Vc\pfibrado{b}{l}\Hc)\pfibrado{\Psi}{\Phi}(\Hc\pfibrado{r}{t}
\Vc)$, to the first an second components are surjective
submersions and the same is true for the projections from the
fiber products $\Vc \pfibrado{b}{l} \Hc$ and $\Hc \pfibrado{r}{t}
\Vc$, then the top, bottom, left and right maps from $\square(\D,
j,i)$ are surjective submersions and the same for the double
source map. It is clear that the compositions, the identities maps and the
inversions maps are smooth.
\end{proof}

\section{Diagonal groupoid associated to a slim double Lie groupoid}

From now on and until Lemma \ref{correspondence} all groupoids are
discrete.

\subsection{Diagonal groupoid}\label{thin-gpd}

In this section we recall from \cite{AN3} the construction of the
diagonal groupoid. Let $\B$ be a double groupoid that satisfies the
filling condition and let $\Vc\circledast\Hc$ be the free product (over $\Pc$)
of the vertical and horizontal groupoids, see \cite{AN3} and \cite{higgins}.
If $\cajaSmall{$A$}{$x$}{$g$}{$y$}{$h$} \in \B$ we denote $ [A] :=
x g y^{-1} h^{-1} \in \Vc\circledast\Hc$. Then  $J_\circledast(\B)$
is the subgroupoid of $\Vc\circledast\Hc$ generated
by $\{[A]| A \in \B\}$.
As $s_\circledast([A]) = e_\circledast([A]) = tl (A)$ we have that the
groupoid $J_\circledast(\B) \tto \Pc$ is in fact a group bundle.
We know that the group bundle $J_\circledast(\B)$ is a a
normal subgroupoid of $\Vc \circledast \Hc$ \cite[lemma 3.5]{AN3}.

Assume that $(\B; \Vc,\Hc; \Pc)$ is slim; then the associated {\em
diagonal groupoid} is $\D(\B)= \Vc\circledast\Hc /
J_\circledast(\B) $. If we compose the natural inclusions of $\Vc$
and $\Hc$ in $\Vc \circledast\Hc$ with the projections on $\D(\B)$
we get two groupoid morphisms:
$$ i: \Hc \to \D(\B) \quad \text{and} \quad j:\Vc \to \D(\B).$$
Thus we have a diagram  $(\D(\B), i, j)$. Our aim is to give
another presentation of the diagonal groupoid as a quotient of
$\Vc \pfibrado{b}{l}\Hc$.

\begin{prop}\label{diagofibr}
Let $(\B; \Vc, \Hc; \Pc)$ be a slim double groupoid that satisfies
the filling condition. We define on $\Vc
\pfibrado{b}{l}\Hc$ the following relation $\sim_\B$:
\begin{center}
$(v_1, h_1)\sim_\B(v_2, h_2)$ if and only if $r(h_1) = r(h_2)$, $t(v_1) = t(v_2)$ and
$v_1h_1h_2^{-1}v_2^{-1} \in J_\circledast(\B)$.
\end{center}
then  $\sim_\B$ is an equivalence relation and the map
\begin{alignat*}3
\phi: \Vc \pfibrado{b}{l}\Hc /\sim_\B\,\; \; \to \D(\B), \qquad
\;[v,h] \mapsto j(v)i(h)
\end{alignat*}
is well defined and is a bijection (of quivers over $\Pc$).
\end{prop}
\begin{proof} Clearly, $\sim_\B$ is an equivalence relation;
we denote  $\G := (\Vc \pfibrado{b}{l}\Hc)/\sim_\B$. Let $[f_1,
x_1] = [f_2, x_2]$ both in $\G$, then $f_1x_1x_2^{-1}f_2^{-1} \in
J_\circledast(\B)$, so $\overline{f_1} \; \overline{x_1} =
\overline{f_2} \; \overline{x_2}$ in $\D(\B)$; where
$\overline{w}$ denotes the image of $w$ under $i$ if $w$ belongs
to $\Hc$, or under $j$ if $w$ belongs to $\Vc$. This proves that
$\phi$ is well defined.
Suppose  that $g,g' \in \Vc$, $x,x' \in \Hc$ is any collection satisfying
$\overline{g}\overline{x} = \overline{g'}\overline{x'}$.
Then $g'x'x^{-1}g^{-1} \in J_\circledast(\B)$, hence $[g', x'] = [g, x]$.
Therefore $\phi$ is injective.

To prove that $\phi$ is surjective, let $d\in\D(\B)$.
Then $d=\overline{d_1}\;\overline{d_2}...\overline{d_n}$
with $d_i$ an element of $\Vc$ or $\Hc$. Let ${d_i} \in \Hc$,
 ${d_{i+1}}\in \Vc$ with
$r(d_{i})=t(d_{i+1})$. Since $\B$ satisfies the filling condition,
the corner
\begin{equation}\label{esquina}\begin{matrix} {d_{i}} \quad
\\ \boxur \; {d_{i+1}} \\ \quad \end{matrix},\end{equation}
can be completed to a box in $\B$, \emph{i.~e.} there exists  $B\in \B$
such that
$$
B = \begin{matrix} \quad {d_{i}} \quad \\ f \,\, \boxe \,\,{d_{i+1}} \\
\quad y\quad \end{matrix}.
$$
Thus, ${d_{i}}{d_{i+1}}y^{-1}f^{-1} \in J_\circledast(\D)$ and
$\overline{d_{i}}\;\overline{d_{i+1}}=
\overline{f}\;\overline{y}$. So, we can commute the $d_i$'s  in
$d=\overline{d_1}\;\overline{d_2}...\overline{d_n}$ in such a way
that we can obtain $d = \overline{g}\;\overline{x}$ with $g \in
\Vc$ and $x \in \Hc$ $b(g)=l(x)$. This proves that
$\phi$ is surjective.
\end{proof}

\begin{obs} \label{estructuradiagonal} $\phi$ induces a structure of groupoid
on  $ \G = \Vc \pfibrado{b}{l}\Hc /\sim_\B$ by:
\begin{itemize}
\item the source and the target projections are
$$s: \G \to \Pc,  \quad [v ,h]  \mapsto t(v); \quad\quad e: \G \to \Pc, \quad [v, h]  \mapsto  r(h).$$
\item The inclusion map is $ \id: \Pc \to \G, \quad p\mapsto \id_p = [\id_p, \id_p].$
\item The partial multiplication is $[v_1,h_1][v_2,h_2] = [v_1f, z\;h_2]$ where $\begin{matrix} \quad h_1 \quad \\ f \,\, \boxe \,\, v_2 \\
\quad z \quad \end{matrix} \in \B.$
\item The inverse is $[v,h]^{-1} = [f^{-1}, z^{-1}] \quad
\text{where} \quad
\begin{matrix} \quad z \quad \\ v \,\, \boxe \,\, f \\
\quad h \quad \end{matrix} \in \B$.
\end{itemize}

If $\Vc$ and $\Hc$ are Lie groupoids, then $b$ and $l$ are
surjective submersions, thus  $\Vc \pfibrado{b}{l}\Hc$ is an
embedded submanifold of $\Vc \times \Hc$. We will prove in Theorem
\ref{diagonalisLie} that $\Vc\pfibrado{b}{l}\Hc /\sim_\B$ is a
Lie groupoid under certain conditions.
\end{obs}

Now we recall a lemma very useful for our purposes.
\begin{lema}\cite[Lemma 3.8]{AN3}\label{tecnico} Let $(\B; \Vc, \Hc; \Pc)$ be a slim double
groupoid that satisfies the filling condition. Let  $f \in \Vc$ and  $x \in \Hc$ such that:
\begin{itemize}
    \item $l(x) = b(f)$ and $t(f) = r(x)$,
    \item  $fx  \in J_\circledast(\B) \subset \Vc \circledast
    \Hc$.
\end{itemize} Then there exists $E \in \Ec(\B)$ such that $E = \begin{matrix} f
\hspace{-0.2cm} &
\begin{tabular}{|p{0,1cm}||} \hhline{|=||} \\ \hline \end{tabular}
\\  & x\end{matrix}$.
\qed
\end{lema}

By this lemma, we have an alternative description of $\sim_\B$.
Indeed, $(f_1, x_1)\sim_\B(f_2, x_2)$ if and only if
$f_1x_1x_2^{-1}f_2^{-1} \in J_\circledast(\B)$. Because
$J_\circledast(\B)$  is a normal subgroupoid, we have that $(f_1,
x_1)\sim_\B(f_2, x_2)$ if and only if $f_2^{-1}f_1x_1x_2^{-1} \in
J_\circledast(\B)$. Hence
\begin{equation} \label{eq:tecnica}
(f_1, x_1)\sim_\B(f_2, x_2) \quad \Longleftrightarrow\qquad \text{there exist} \quad E =
{\tiny\begin{matrix} f_2^{-1}f_1 \hspace{-0.2cm} &
\begin{tabular}{|p{0,1cm}||} \hhline{|=||} \\ \hline \end{tabular}
\\  & x_1x_2^{-1}\end{matrix}} \quad \in \B.
\end{equation}
Thus, the graph of the relation $\sim_\B$ is
\begin{align*}
R = \{(f_1, x_1, f_2, x_2) \in (\Vc \pfibrado{b}{l}\Hc)
\pfibrado{\eta}{\eta} (\Vc \pfibrado{b}{l}\Hc) |  \;\exists E\in
 \Ec(\B) , \; E = {\tiny\begin{matrix} f_2^{-1}f_1
\hspace{-0.2cm} &
\begin{tabular}{|p{0,1cm}||} \hhline{|=||} \\ \hline \end{tabular}
\\  & x_1x_2^{-1}\end{matrix}}\}
\end{align*}
where
\begin{equation*}
\eta: \Vc \pfibrado{b}{l}\Hc \to \Pc \times \Pc, \quad (f, x)
\mapsto (t(f), r(x)).
\end{equation*}
We conclude that the relation $\sim_\B$ is determined by the core
groupoid of $\B$.
\begin{lema}\label{correspondence}Let $(\B; \Vc, \Hc; \Pc)$ be a slim
double groupoid satisfying the filling condition.
If $(f_1,x_1),(f_2,x_2)  \in \Vc \pfibrado{b}{l}\Hc$, then
$(f_1,\;x_1)\sim_\B(f_2,\; x_2)$ if and only if there exist $A, B
\in \B$ such that $$A =
\begin{matrix} \quad x \quad \\ f_1 \,\, \boxe \,\, g \\
\quad x_1 \quad \end{matrix}, \quad \text{and} \quad B =
\begin{matrix} \quad x \quad \\ f_2 \,\, \boxe \,\, g \\
\quad x_2 \quad \end{matrix}.$$
\end{lema}
\begin{proof} In fact, if $$A =
\begin{matrix} \quad x \quad \\ f_1 \,\, \boxe \,\, g \\
\quad x_1 \quad \end{matrix}, \quad \text{and} \quad B =
\begin{matrix} \quad x \quad \\ f_2 \,\, \boxe \,\, g \\
\quad x_2 \quad \end{matrix} \quad \text{are in} \quad \B, $$ then
$xgx_1^{-1} f_1^{-1} \in J_\circledast(\B)$ and
$xgx_2^{-1}f_2^{-1} \in J_\circledast(\B)$, taking inverse of the
first and composing, it follows that $f_1x_1x_2^{-1}f_2^{-1} \in
J_\circledast(\B)$, \emph{i.~e.} $(f_1,\;x_1)\sim_\B(f_2,\; x_2).$

Reciprocally, if $(f_1,\;x_1)\sim_\B(f_2,\; x_2)$ then
by (\ref{eq:tecnica}), there is $E \in \Ec(\B)$ such that
$E = \begin{matrix} f_2^{-1}f_1 \hspace{-0.2cm} &
\begin{tabular}{|p{0,1cm}||} \hhline{|=||} \\ \hline \end{tabular}
\\  & x_1x_2^{-1}\end{matrix}.$
The filling condition guarantees that given $x_2 \in \Hc$ and $f_2 \in
\Vc$ with $l(x_2) = b(f_2)$, there is a box $B' \in \B$ with
$t(B')= x_2$ and $l(B') = f_2^{-1}$. Let $B'  =
\begin{matrix} \quad \;\;x_2 \quad \\ f_2^{-1} \,\, \boxe \,\, g^{-1} \\
\quad \;\;x \quad \end{matrix}$, and let
$A':= \left\{\begin{matrix} E^v & \idd(x_2) \vspace{-1pt}\\
\idd(f_2) & B' \end{matrix} \right\} =
\begin{matrix} \quad \;\; x_1\quad \\ f_1^{-1} \,\, \boxe \,\, g^{-1} \\
\quad \;\;x \quad \end{matrix}.$
Let $A$, $B$ be the vertical inverses of $A'$
and $B'$ respectively. Thus
$A = \begin{matrix} \quad x \quad \\ f_1 \,\, \boxe \,\, g \\
\quad x_1 \quad \end{matrix}$ and $B =
\begin{matrix} \quad x \quad \\ f_2 \,\, \boxe \,\, g \\
\quad x_2 \quad \end{matrix}$ are both in $\B$ and we get the result.
\end{proof}

\subsection{The core action}

We recall that a continuous map $f:X \to Y$, between two
topological spaces $X$ and $Y$, is said to be \textit{proper}
if the inverse image of a compact subset of $Y$ is compact.
\begin{definition}
A Lie groupoid $\G \tto \Pc$ is  \textit{proper} if the anchor map
is proper.
An action of a Lie groupoid $\G$ on a smooth manifold $Z$ is
\textit{proper} if  the action groupoid $\G \ltimes Z$ is proper.
\end{definition}

The following proposition is useful to decide when an action is
proper. For details and more on proper actions, see \cite{tu}.

\begin{prop}\label{prop-tu}
Let $\G \tto \Pc$ be a Lie groupoid. Let $Z$ be a smooth manifold
endowed with a left action of $\G$, then $\G$ acts properly on $Z$
iff the anchor map $(s, e): \G \ltimes Z \to Z \times Z$ is closed and $\forall z
\in Z$, the stabilizer of $z$ is compact. \qed
\end{prop}

We shall need the following proposition from \cite[Prop. 1.1]{AN3}.

\begin{prop}\label{propactioncore}Let $(\B; \Vc, \Hc; \Pc)$ be a slim double
groupoid. Define $\gamma:\B \to \Pc$,  $\gamma(A)= lb(A)$, the
bottom-left vertex of $A$.

(a). There is an action of the core groupoid $\Ec(\B)$ on $\gamma:
\B \to \Pc$ given by \begin{equation}\label{eqn:actioncore} E
\acts A : = \left\{\begin{matrix}\iddv l(A)& A \vspace{-4pt}\\ E
&\iddv b(A)\end{matrix} \right\}, \quad A\in \B, E \in \Ec.
\end{equation}

\medbreak (b). Let $B\in \B$. Then  the stabilizer $\Ec(\B)^B$ is
trivial and the orbit of $B$ is $\Oc_B = \{ A \in \B: t(A) = t(B),
r(A) = r(B) \}$. \qed \end{prop}

The above results enable us to state and proof the following
lemma.

\begin{lema}\label{actioncore} Let $(\B;\Vc,\Hc;\Pc)$ be a slim (discrete)
double groupoid. Define the map $\eta : \Vc \pfibrado{b}{l} \Hc \to
\Pc$ by $\eta(f,x)=b(f) = l(x)$. Then $\Ec(\B)$ acts on $\eta$ by
\begin{equation}\label{eqn:actioncoreexplicxita}
E \rhd (f,x) = ( f \; l(E), b(E) \;x), \quad\text{ when }\eta
(f,x) = e_{\Ec(\B)}(E).
\end{equation}

The quotient space $\Vc \pfibrado{b}{l} \Hc / \Ec(\B)$ coincides
with $\D(\B)$ .

\end{lema}
\pf Since $b (f \; l(E)) = bl(E) = s_{\Ec(\B)}(E)$ and
$l(b(E) \; x) = lb(E) = s_{\Ec(\B)}(E)$ the map
$\rhd : \Ec(\B) \pfibrado{e_{\Ec(\B)}}{\eta}(\Vc \pfibrado{b}{l} \Hc)
 \to \Vc \pfibrado{b}{l} \Hc$  is well
defined. That $\rhd$ is an action is straightforward, in fact,
\begin{align*}
(E \circ F) \rhd (f, x)&=(f\; l(E \circ F),b(E \circ F) \; x)
= ( f \; l(F)l(E),b(E)b(F) \; x) \\
  &= E \rhd (f \; l(F),t(F) \, x) = E \rhd ( F \rhd (f, x)).
  \end{align*}
Also, $ \eta(E \rhd (f,x)) = \eta(f\;l(E), b(E)\;x) = l(b(E)x) =
lb(E))=e_{\Ec(\B)}(E)$.

For the second part, if $(f, x)\sim_{\B} (g, y)$, then there are
$A,B \in \B$ such that \cajaSmall{$A$}{$z$}{$h$}{$x$}{$f$} and
\cajaSmall{$B$}{$z$}{$h$}{$y$}{$g$}, see Lemma
\ref{correspondence}. Then, by Proposition \ref{propactioncore},
there exists a box $E \in \Ec(\B)$ such that $A = E
\rightharpoondown B$, in consequence, $x = b(E)y$ and $f = g\,
l(E)$.

Conversely, if $A, \;B \in \B$ and there exists $E \in \Ec(\B)$
with $b(A) = b(E)b(B)$ and $l(A) = l(B) l(E)$, then the
boxes $E \rightharpoondown B$ and $B$ have the same
top and right sides. By Lemma
\ref{correspondence}, we have $(l(E \rightharpoondown B), b(E
\rightharpoondown B)) \sim_{\B} (l(B), b(B))$, that is $(l(A),
b(A))\sim_{\B}(l(B), b(B))$. From this we conclude that given $(f,
x)$, $(g, y) \in \Vc \pfibrado{b}{l} \Hc$,  $(f, x)\sim_{\B}(g,
y)$ iff $\exists E \in \Ec(\B)$ with $f = g\;l(E)$ and $x = b(E)
y$. Thus the quotient coincides with the diagonal groupoid. \epf

The action \eqref{eqn:actioncoreexplicxita} will be called the
\textit{core action} of $\Ec(\B)$ on $\Vc \pfibrado{b}{l}
\Hc$. Let  $\pi:\Vc \pfibrado{b}{l} \Hc \to \Vc \pfibrado{b}{l}
\Hc/ \Ec(\B)$ be the projection  determined by
\eqref{eqn:actioncoreexplicxita}.

\begin{theorem}\label{diagonalisLie}
Let $(\B; \Vc, \Hc; \Pc)$ be a slim double Lie groupoid. If the
\textit{core action} is proper, then $\D(\B)$ is a Lie
groupoid over $\Pc$.
\end{theorem}
\pf Since $\B$ is slim, the action
\eqref{eqn:actioncoreexplicxita} is free. Hence, if the action is
proper, then the quotient $\Vc \pfibrado{b}{l} \Hc / \Ec(\B)$ has
a unique manifold structure such that the projection $\pi:\Vc
\pfibrado{b}{l} \Hc \to \Vc \pfibrado{b}{l} \Hc / \Ec(\B)$ is a
surjective submersion \cite[Theorem 3.3.1]{d}. Thus, $\D(\B)$ is a Lie groupoid over
$\Pc$. In fact, the structure maps are described in Remark
\ref{estructuradiagonal}; using local sections of $\pi$, it is
clear that the source and target maps are surjective submersions
and that the other structural maps are smooth. \epf

Let $(\B,\Vc,\Hc,\Pc)$ be a slim double Lie groupoid. Let
\begin{align*}
\tilde{i}: \Hc \to \Vc \pfibrado{b}{l} \Hc \quad & &x \mapsto ( \id(l(x)), x) \\
\tilde{j}: \Vc \to \Vc \pfibrado{b}{l} \Hc \quad & &f \mapsto
(f,\id(b(f))),
\end{align*}
 be the canonical inclusions of $\Vc$ and $\Hc$; let also $i = \pi \circ \tilde{i}$ and $j = \pi
\circ \tilde{j}$.

\begin{lema}\label{transersality}
The maps $i$ and $j$ defined above are transversal at the
identities.
\end{lema}

\pf Let $p \in \Pc$. Take a tangent vector $Z \in T_{[\id_p,
\id_p]} \D(\B)$, where $[\id_p, \id_p] = \pi(\id_p, \id_p)$. Since
$\pi$ is a surjective submersion, there is $(U, W) \in T_{(\id_p,
\id_p)}(\Vc \pfibrado{b}{l} \Hc)$ such that $T_{(\id_p, \id_p)}
\pi (U, W) = Z$. Choose
\begin{align*}
Y = U \in T_{\id_p} \Vc, \qquad X = W - (T_p \id_{\Hc})(T_{ \id_p}
b)(U) \in T_{ \id_p} \Hc.
\end{align*}
It is clear that $$(T_{ \id_p} \tilde{j}\;)(Y) = (U, (T_p
\id_{\Hc})(T_{ \id_p} b(U)))\text{ and }(T_{ \id_p}
\tilde{i}\;)(X) = ((T_p \id_{\Vc})(T_{ \id_p} l)(X), X).$$ We
compute
\begin{align*}
(T_p \id_{\Vc})(T_{\id_p} l)(X) &=(T_p
\id_{\Vc})(T_{ \id_p} l)(W) -
(T_p \id_{\Vc})(T_{\id_p} l)(T_p \id_{\Hc})(T_{\id_p} b)(U)\\
&=(T_p \id_{\Vc})(T_{ \id_p} l)(W) - (T_p \id_{\Vc})(T_{\id_p}
b)(U) =0;
\end{align*}
then $(T_{\id_p} \tilde{i}\;)(X) = (0, X)$. In consequence
we have

\begin{align*}
(T_{\id_p} \tilde{j}\;)(Y) + (T_{\id_p} \tilde{i}\;)(X) &= (U, (T_p \id_{\Hc})(T_{\id_p} b)(U)) + (0, X)\\
&=(U, (T_p \id_{\Hc})(T_{\id_p} b)(U)+ W - (T_p \id_{\Hc})(T_{\id_p} b)(U))\\
&=(U, W).
\end{align*}
Then if we apply  $T_{(\id_p,\id_p)} \pi$ to both
sides of the above equation we arrive to
$$(T_{\id_p} j)(Y) + (T_{\id_p} i) (X) = Z,$$
that is, the maps $i$ and $j$ are transversal at the identities.
\epf

\medbreak
Let $(\D, j, i)$ be a $(\Vc, \Hc)$-factorization. The underlying
manifold to the core groupoid of $\B = \square(\D, j, i)$ is $\Vc^{op}
\pfibrado{j}{i} \Hc = \{(h, y) \; | \; j(h^{-1}) = i(y)\}$.
The core action on $\Vc \pfibrado{b}{l} \Hc$ is given by
\begin{equation}\label{actioncorefactorization}
(h, y) \rhd (f, x) = (fh, yx) \quad \text{when} \quad \eta(f,x)=
t(h) = r(y);
\end{equation}
the proof of \eqref{actioncorefactorization} follows
from the definition \eqref{eqn:actioncoreexplicxita}.
\begin{lema}\label{actdougroupdiag}
The core action \eqref{actioncorefactorization} is proper.
\end{lema}
\pf Since the action \eqref{actioncorefactorization} is free,
in order to prove that it is proper, we only need to
check that the anchor map of the respective action groupoid
$$(s, t) : (\Vc^{op} \pfibrado{j}{i} \Hc) \ltimes (\Vc \pfibrado{b}{l}
\Hc) \to (\Vc \pfibrado{b}{l} \Hc) \times (\Vc \pfibrado{b}{l}
\Hc)$$ is closed, see Proposition \ref{prop-tu}. Let $A \subseteq
(\Vc \pfibrado{j}{i} \Hc) \ltimes (\Vc \pfibrado{b}{l} \Hc)$ be a
closed set and consider a sequence  $\{(f_n, x_n, g_n, y_n)\}_{n
\in \mathbb{N}}$ in $A$ such that the sequence $$\{ (s, t)(f_n,
x_n, g_n, y_n)\}_{n \in \mathbb{N}} = \{(g_n f_n, x_n y_n, g_n,
y_n)\}_{n \in \mathbb{N}}$$ converges to $(a, b, g, y) \in (\Vc
\pfibrado{b}{l} \Hc) \times (\Vc \pfibrado{b}{l} \Hc)$. We need to
see that $(a, b, g, y) \in (s, t) (A)$. Clearly, $g_n \underset{n
\rightarrow \infty}\longrightarrow g$, \;\;$y_n \underset{n
\rightarrow \infty} \longrightarrow y$, \;\;$g_n f_n \underset{n
\rightarrow \infty} \longrightarrow a$ and \;\;$x_n y_n
\underset{n \rightarrow \infty} \longrightarrow b$.

\medbreak Hence $(f_n, x_n, g_n, y_n) \underset{n \rightarrow
\infty} \longrightarrow (g^{-1}a, b y^{-1}, g, y)$; since $A$ is
closed, we conclude that $((g^{-1}a, b y^{-1}, g, y) \in A$. Now
$(a, b, g, y) = (s, t)(g^{-1}a, b y^{-1}, g, y) \in (s, t)(A)$ by
a direct calculation, hence $(s, t)(A)$ is closed. \epf

Finally, we arrive to our main result.

\begin{theorem}\label{categoryequivalence}
Fix $\Vc$ and $\Hc$. The assignments $\B \mapsto \D(\B)$ and $(\D,
j, i) \mapsto \square(\D, j, i)$ determine mutual category
equivalences between
\begin{enumerate}
\item [(a)]  The category of slim double Lie groupoids $(\B; \Vc, \Hc; \Pc)$
with proper core action, and

\item [(b)] The category of $(\Vc, \Hc)$-factorizations of Lie groupoids $(\D,
j, i)$ such that the maps $i$ and $j$ are transversal at the
identities.

\end{enumerate}

\end{theorem}

\pf The equivalence of categories at the discrete level was
obtained in \cite{AN3}.

If $\B$ is a double Lie groupoid  as in \textit{(a)}, then the
associated $(\Vc, \Hc)$-factorization $(\D(\B), j, i)$ is a Lie
groupoid with $i$ and $j$ transversal to the identities, by
Theorem \ref{diagonalisLie} and Lemma \ref{transersality}, respectively.

Conversely, if we begin with a $(\Vc, \Hc)$-factorization of Lie
groupoids as in \textit{(b)}, then the associated double groupoid
is a slim double Lie groupoid as is required in \textit{(a)}, by
Theorem \ref{associatedslim} and Lemma \ref{actdougroupdiag}. \epf


\begin{thebibliography}{99999}


\bibitem[AM06]{AM} {\sc N. Andruskiewitsch} and {\sc J. M. Mombelli},
\emph{Examples of weak Hopf algebras
arising from vacant double groupoids}, Nagoya Math. J. {\bf 181}
(2006).

\bibitem[AN05]{AN1} {\sc N. Andruskiewitsch} and {\sc S. Natale},
\emph{Double categories and quantum groupoids},
Publ. Mat. Urug. {\bf 10}, 11-51 (2005).

\bibitem[AN06]{AN2} \bysame, \emph{Tensor categories attached to double
groupoids}, Adv. Math. {\bf 200}, 539-583 (2006).


\bibitem[AN06b]{AN3}\bysame,
\emph{The structure of double groupoids},   arXiv:math/0602497v3 [math.CT]

\bibitem[B04]{brown} {\sc R. Brown},
\emph{Crossed complexes and homotopy groupoids as non commutative
tools for higher dimensional local-to-global problems}, Fields
Inst. Commun. {\bf 43}, 101--130, Amer. Math. Soc. (2004).

\bibitem[BJ04]{bj} {\sc R. Brown} and {\sc G. Janelidze},
\emph{Galois theory and a new homotopy double groupoid of a map of spaces},
Appl. Categ. Structures {\bf 12},  63--80 (2004).


\bibitem[BM92]{BM} {\sc R. Brown} and {\sc K. Mackenzie},
\emph{Determination of a double Lie groupoid by its core diagram},
J. Pure Appl. Algebra {\bf 80},  237--272 (1992).

\bibitem[BS76]{bs} {\sc R. Brown} and {\sc C. Spencer},
\emph{Double groupoids and crossed modules},
Cahiers Topo. et G\'eo. Diff. {\bf XVII},  343--364 (1976).

\bibitem[D07]{d} {\sc Dragulete, Oana Mihaela} Some applications of symmetries in differential geometry and dynamical systems, Ph.D Theses, �cole Polytechnique F�d�rale de Lausanne, (2007).

\bibitem[E63]{ehr} {\sc C. Ehresmann},  \emph{Cat\' egories doubles et cat\' egories structur\' ees},
C. R. Acad. Sci. Paris {\bf 256}, 1198--1201 (1963). \emph{Cat\' egories structur\' ees}  Ann. Sci. \' Ecole Norm. Sup. {\bf 80},  349--426 (1963).

\bibitem[H71]{higgins} {\sc P. J. Higgins},
\emph{Categories and groupoids},
 Repr. Theory Appl. Categ.  {\bf 7}  1--178 (2005). Reprint of the 1971
 \emph{Notes on categories and groupoids}, Van Nostrand Reinhold, London.

\bibitem[L02]{Lang}{\sc S. Lang}, \emph{Introduction to differentiable manifolds}, Springer-Verlag New York (2002).

\bibitem[LW89]{wl} {\sc  J.-H. Lu and A. Weinstein},
\emph{Groupo\" ides symplectiques doubles des groupes de Lie-Poisson},
C. R. Acad. Sci. Paris S\' er. I Math. {\bf 309},  951--954 (1989).


\bibitem[M92]{mk1} {\sc K. Mackenzie},  \emph{Double Lie algebroids
and Second-order Geometry, I}, Adv. Math. {\bf 94}, 180--239
(1992).


\bibitem[M99]{mk3} {\sc K. Mackenzie},
\emph{On symplectic double groupoids and the duality of Poisson
groupoids},  Internat. J. Math. {\bf 10},  435--456  (1999).


\bibitem[M00]{mk2} \bysame,  \emph{Double Lie algebroids
and Second-order Geometry, II}, Adv. Math. {\bf 154}, 46--75
(2000).


\bibitem[M05]{mk4} \bysame, \emph{General Theory of Lie Groupoids and
Algebroids}, London Mathematical Society Lecture Note 213,
Cambridge University Press (2005).

\bibitem[MM03]{mm1}{\sc I. Moerdijk and J. Mr\v{c}un}, \emph{Introduction to foliations and Lie groupoids},Cambrige
Studies in Advanced Mathematics 91, Cambridge University Press
(2003).

\bibitem[P77]{p} {\sc  J. Pradines},  \emph{ Fibr\'es vectoriels doubles et calcul des jets non holonomes},
Esquisses Math. {\bf 29}, Universit\'e d'Amiens, Amiens, (1977).

\bibitem[T04]{tu} {\sc J. L. Tu}, \emph{Non-Hausdorff groupoids, proper actions and K-theory}, Documenta Math.
 {\bf 9}, 565--597 (2004).


\end{thebibliography}
\end{document}